\documentclass[11pt,a4paper]{article}
\usepackage{geometry}
\title{Optimal routing and transmission strategies for UAV reconnaissance missions with detection threats}
\author{Riley Badenbroek\footnote{Erasmus University Rotterdam} \and Relinde Jurrius \footnote{Netherlands Defence Academy} \and Lander Verlinde \footnote{School of Computer Science, University of Auckland, New Zealand. Email: \texttt{lver263@aucklanduni.ac.nz}.~Research supported by the New Zealand Marsden Fund.}}

\date{}

\usepackage{amsmath,tikz, amssymb,amsthm, multirow, pgfplots, booktabs, bbm}
\pgfplotsset{compat=1.18}
\usepgfplotslibrary{groupplots}
\newtheorem{theorem}{Theorem}
\newtheorem{proposition}[theorem]{Proposition}
\usepackage[hidelinks]{hyperref}
\usepackage{cleveref}
\usetikzlibrary{decorations.pathreplacing, shapes.arrows, positioning}

\usepackage[english]{babel}

\newcommand{\vertices}{V}
\newcommand{\edges}{E}
\newcommand{\timeperiods}{T}
\newcommand{\timemax}{T_{\text{max}}}
\newcommand{\survmove}{s^{\text{M}}}
\newcommand{\survsend}{s^{\text{S}}}
\newcommand{\expmove}{v^{\text{M}}}
\newcommand{\expobs}{v^{\text{O}}}
\newcommand{\expsend}{v^{\text{S}}}
\newcommand{\base}{0}
\newcommand{\varmove}{x}
\newcommand{\varobs}{y}
\newcommand{\varsend}{z}
\newcommand{\probmove}{q}
\newcommand{\probsend}{p}
\newcommand{\valinfo}{w}
\newcommand{\survsendvarmove}{\alpha}
\newcommand{\survmovevarsend}{\beta}
\newcommand{\expsendvarmove}{\gamma}
\newcommand{\survmovevarobs}{\delta}
\newcommand{\expobsvarsend}{\epsilon}
\newcommand{\bigM}{M}

\begin{document}

\maketitle

\begin{abstract}
We consider an autonomous reconnaissance mission where an Unmanned Aerial Vehicle (UAV) has to visit several points of interest and communicate the intel back to the base. At every point of interest, the UAV has the option to either send back all available info, or continue to the next point of interest and communicate at a later stage. Both choices have a chance of detection, meaning the mission fails. We wish to maximize the expected amount of information gathered by the mission. This is modelled by a routing problem in a weighted graph. We show that the problem is NP-complete, discuss an ILP formulation, and use a genetic algorithm to find good solutions for up to ten points of interest.\\ 

\noindent Keywords: linear programming, genetic programming, graph theory, mission planning, vehicle routing

\noindent MSC: 90C90, 90C59
\end{abstract}



\section{Introduction}

The role of Unmanned Aerial Vehicles (UAVs), more commonly known as drones, in society continues to become more significant every day. The civil market alone is estimated to be worth 7.2 billion USD in 2022 and this value is expected to grow to 19.2 billion USD in 2031 \cite{DroneMarket}.
Applications range from agriculture over disaster response to package deliveries. But also its military use has become more relevant. Already in the Vietnam War, the US army deployed drones as a weapon \cite{Vietnam}. When surveillance technology improved, it became clear that drones could also be used for survey missions in enemy terrain \cite{Marine}. Moreover, these automated missions need not be performed by aerial systems, depending on the terrain and the characteristic of the mission, an automated ground vehicle (UGV) or an unmanned under water vehicle (UUV) may be more adept. 
The wars in Ukraine and Gaza have once again shown the importance of hybrid warfare \cite{HybridUkraine}. In such warfare, the distinction between different modes of warfare (conventional warfare, cyber warfare, political warfare) tend to blurred. This makes the adversary in such war fluid and harder to predict \cite{HybridConcept}. Intelligence in physical and non-physical infrastructure is key in gaining advantage. Hence the extent to which unmanned vehicles are used for both offensive as well as reconnaissance missions is at an all-time high \cite{DronesWarfare}.

To expand the number of operational systems while managing costs, it is desirable to deploy systems that can operate fully independently. For a reconnaissance mission, this requires a planning of the complete mission before the drone leaves for enemy territory. This is a so-called \textit{Single-Hop Routing Protocol}, as the the planning is updated in only one `hop' \cite{jiang2018routing}. The setting of such a mission can be stated as follows: starting from a secure base, multiple surveillance locations need to be safely reached and the acquired information has to be brought back to the base. There are three possible ways of bringing back information. At every surveillance location, there is the possibility of transmitting information back to the base camp. Hence the first option is that the drone travels to a surveillance location, gathers the info and immediately sends it back to the base. Secondly, the drone could also store the information and go to the next surveillance location. After having obtained the information there, the totality of info could then be transmitted together at that location. Thirdly, the drone could also store information and return to the base. In this case the information is physically obtained back from the drone. However, each action in the mission carries the risk of detection – the drone could be spotted during flight, or transmissions might be intercepted. Both ways of detection give away the position of the drone, ending the mission abruptly.

This paper investigates how to find the optimal strategy of these reconnaissance missions. Such an optimal strategy consists of two elements: both the route and the send strategy have to be optimal to maximize the amount of retrieved information. Hence the specific questions for which an answer is proposed in this research are:

\begin{itemize}
    \item In which order should the different locations be travelled to?
    \item Where is it beneficial to make a transmission, and where is it better to hold on to the gathered information?
\end{itemize}

\subsection{Related Work}
Mission planning for UAV reconnaissance has been studied extensively in literature, as there are various factors to take into account. These factors can be subdivided into three categories: \textit{modelling of means}, \textit{modelling of terrain} and \textit{modelling of threat} \cite{song2019survey}. Every such category yields a different objective in the optimisation and a different way of approaching the problem. 

\textit{Modelling of means} encompasses all specific and intrinsic vehicle characteristics. A UAV has different features than a UGV, and hence needs to be modelled differently. Evidently, UAVs are able to move up and down, giving an additional dimension to the routing problem \cite{song2019survey}. Furthermore, a small UAV could require a recharge during the mission, maybe even on a mobile charging station. This scenario was studied and modelled by Yu et al. \cite{yu2019algorithms}. But also other types of missions combine both types of unmanned vehicles; see Lu et al. \cite{lu2023uav} for a review. Modelling of means also comprises of the specific goals besides surveillance of the mission. Sajid et al. \cite{sajid2022routing} propose a model for missions which also contain a delivery component.

Secondly, \textit{modelling of terrain} comprises of all terrain elements the vehicle can encounter during the mission. Especially UGVs are constrained by the terrain. Roberge, Tarbouchi and Labonté \cite{roberge2012comparison} 
have divided the terrain into small square cells which represent areas of similar altitude. They then study the performance of genetic algorithms for routing across such a grid, taking into account fuel constraints coming from uphill and downhill movements. Fink et al. \cite{fink2019globally} have proposed an adaptation of Dijkstra's algorithm for multi-objective routing in three-dimensional, mountainous terrain. Interestingly, they do so for routing a rover on Mars. But also when considering UAVs, the terrain plays a role. There might be altitude constraints, or constraints related to the angle of ascent and descent, especially when multiple UAVs are involved and collisions should be avoided; see Jia et al. \cite{jia2020dynamic} for an approach of this scenario. Furthermore, Albert, Leira and Imsland \cite{albert2017UAV} considered the problem where multiple UAVs surveil moving targets. They propose an Mixed Integer Linear Program to model this situation and do a case study based on moving icebergs. 

Lastly, \textit{modelling of threats} includes all elements with an incentive to jeopardize the mission, e.g. radars and enemy lookout posts. Alotaibi et al. \cite{alotaibi2018unmanned} model these threats as a cost function on the edges of a graph based on the distance to hostile waypoints. Dasdemir, Köksalan and Öztürk \cite{dasdemir2020flexible} consider a continuous terrain with circular radar zones. Another possible threat is jamming, making it impossible to send any information back to the basis. Han et al. \cite{han2021satellite} propose a reinforcement learning approach to deal with an uncertain jamming state when routing. There are multiple ways to deal with threats -- proactive or reactive -- which require different algorithms; see also Quadt et al. \cite{quadt2024dealing} for a recent and extensive survey.

\subsection{Our Contribution}
The main contribution of this paper is considering a novel characteristic in the model of means which implies a new threat to a reconnaissance mission. In a classical setting, it is assumed that the unmanned vehicle physically brings gathered information back to the basis, or the risk of detection when transmitting this info is not considered. In this work, we consider what happens when we do take into account the ability of transmission and the threat of interception. Hence, we assume that, based on the vehicle characteristics, modelling of the terrain, and similar techniques as in Alotaibi et al. \cite{alotaibi2018unmanned}, we have been given a graph with a cost function on the edges corresponding to the survival probability to cross that edge. But now we also add a cost function to the vertices, corresponding to surviving a potential transmission at that location. This added threat requires a new model and new solutions, which are presented below.

The paper is structured in five sections. After this introduction, the problem is described mathematically. A model based on weighted graphs is proposed and two ways to compute the expected value of retrieved information are discussed. Furthermore, it is investigated whether the problem can be written as an Integer Program and a motivation is given for the choice of a heuristic algorithm. In the third section, the case where one drone is deployed during the mission is examined in great detail. A genetic program that yields the optimal strategy is presented. This genetic program is tested on several mission scenarios and different algorithms are compared in terms of success rate and complexity. In the last section the scope is extended to missions with multiple drones. The best genetic program from the single drone scenario is adapted and improved such that it can also solve the multiple drone scenario. This improved algorithm is tested in the same way as the single drone scenario and a comparison is made.

\section{Mathematical exploration of the problem}

The objective for a planning of an autonomous reconnaissance mission is to maximize the expected value of the transmitted information. This section investigates how to model such a mission mathematically and how to compute the expected value. 

A mission on $n$ surveillance locations can be naturally formulated as a problem on a weighted graph $G= (V, E)$, with $V = \{0, \dots, n-1\}$. The weights consist of both edge weights and vertex weights. 

Every edge $\{i,j\}$ has weight $q_{ij} \in [0,1]$. These edge weights correspond to the crossing probabilities: the survival chance of crossing from $i$ to $j$. Moreover, every vertex $i$ has two weights. The first one $p_i \in [0,1]$ is the transmission probability: the probability of making a successful transmission at vertex $i$. The second weight $w_i$ is the amount of information that can be gathered at that vertex. Having these weights and probabilities, finding the optimal strategy of the mission, boils down to finding a walk with corresponding send strategy that maximizes the expected amount of retrieved information. In such a strategy, the following rules apply:
\begin{itemize}
    \item Vertex 0 corresponds to a safe base camp: successful transmission has probability one, but there is no information to retrieve here. This means that $p_0$ is equal to 1 and $w_0$ is 0.
    \item The walk can have repeated vertices. However, once the information has been retrieved at a vertex $i$, $w_i$ is set to 0. 
    \item If information is retrieved but not transmitted, it is carried to the next vertex in the walk and can be transmitted there or further down the walk. A transmission always sends all information that has been retrieved but not yet transmitted. At the last vertex of the walk -- the base camp -- a transmission is always made, but this transmission can be empty.
    \item Once a crossing or a transmission fails, the mission is over and no more information can be retrieved nor transmitted.
\end{itemize}

Following the above rules, we formulate the expression for the expected value of transmitted information. By considering the different vertices where information is transmitted, we compute how much information every transmission is expected to contain. Let $R = \left[v_1, v_2, \dots, v_k \right] $ with $v_1 = v_k=0$ be the sequence of vertices that make up the walk. Say $v_i$ has transmission probability $p_{v_i}$ and the probability of crossing from $v_i$ to $v_{i+1}$ is $q_{v_iv_{i+1}}$. Moreover, let $S = \left[v_{s_0}, v_{s_1}, v_{s_2}, \dots, v_{s_m}\right]$ be the subsequence of $R$ consisting of $v_{s_0} = v_1$, followed by the $m$ vertices where information is transmitted. Then we compute the probability to survive the entire route with corresponding send strategy with

\[ \mathbb{P}(\text{survival}) =  \prod_{j=1}^{k-1} q_{v_jv_{j+1}} \prod_{i=1}^{m} p_{s_i} .\]

And if we let $X$ denote the random variable of the amount of well-received information, then the total expected information is given by

\[ \mathbb{E}[X] =\sum_{i=1}^{m} \prod_{j=0}^{s_i -1} q_{v_j v_{j+1}} \prod_{ \substack{u: v_{s_u} \in S \\ u \leq i}} p_{s_u} \sum_{ s_{i-1} < h \leq s_i} w_h. \]

\subsection{An example of a reconnaissance mission}\label{subsec: first example}
To get a better insight in the above formula and in to why this problem is harder than initially expected, it is useful to look at an example. Consider a mission which is given by the graph in \autoref{fig: example graph}. This graph consists of a base camp at vertex 0 and three surveillance locations at vertices 1, 2 and 3. The crossing and transmission probabilities are shown on the graph. At every surveillance location there is one unit of information to be retrieved, so $w_i=1$ for all $i$.
Now let's assume that the following strategy is chosen:
    \begin{align*}
    \text{Route: } &(0,1,2,3,0) \\
    \text{Send: }  &(0,1,0,0,1).
    \end{align*}
This send strategy means that a transmission is made at vertex 1 and back in the base camp upon return.

Let $X$ be the random variable for the amount of well-received information for this route and send strategy. $X_1$ is the random variable that defines the amount of received information at the first transmission and $X_2$ at the second one.
By linearity of expectation: $\mathbb{E}[X] = \mathbb{E}[X_1]+ \mathbb{E}[X_2]$. 
To compute $\mathbb{E}[X_1]$, first the probability of safely reaching vertex 1 needs to be computed, which is equal to 0.6. The transmission probability is 0.9 and there is one unit of information being transmitted. So:
$$ \mathbb{E}[X_1] = 0.6\cdot0.9\cdot1 = 0.54. $$
We make a similar computation for the second transmission. The probability of reaching the base camp is $0.6\cdot0.9\cdot0.3\cdot0.9\cdot0.6$. The transmission probability is 1 and there are two units of information transmitted: the information from vertices 2 and 3. The expected amount of transmitted information with the second transmission then becomes:
$$\mathbb{E}[X_2] = 0.6\cdot0.9\cdot0.3\cdot0.9\cdot0.6\cdot1\cdot2 = 0.17496.$$

Combining this yields the total expected value:
$$\mathbb{E}[X] = 0.54 + 0.71496 = 0.71496.$$

So for this graph, if the route $(0,1,2,3,0)$ is chosen with transmissions at vertices 1 and 0, the mission is expected to retrieve 0.71496 units of information out of the possible 3 units. In fact, this is the best strategy using a Hamilton cycle in the graph. This might be a bit disappointing when setting up a mission, as we don't even expect a third of all possible information to be retrieved with this strategy. Fortunately, we are able to find an optimal strategy with a higher expected value: 
    \begin{align*}
    \text{Route: } &(0,2,3,2,0,1,0) \\
    \text{Send: }  &(0,0,0,0,1,1,1)
    \end{align*}
has an expected value of 1.666, more than double of the best Hamilton cycle. This shows that the best strategy is not necessarily a `pretty' route that is easy to predict. That is part why this problem is hard to solve. In the first place, the route of the optimal strategy is not always intuitive: it doesn't necessarily have the nice structure of a Hamilton cycle, nor is it necessarily going back and forth between base camp and other vertices. This is due to the fact that our mission graphs are not necessarily metric. Traveling between two connected  vertices $u$ and $v$ might have a higher survival probability by going around via some other vertices instead of via the edge $\{u,v\}$. And secondly, the expected value is dependent on both the route and the send strategy. This means that we cannot first optimise the route and then find the optimal send strategy corresponding to this route -- this would have been possible using a dynamic program. This complicates things when looking for an algorithm to solve this problem.

   \begin{figure}[ht]
        \centering
        \begin{tikzpicture}[scale=2]
        	\node[draw,circle] (0) at (0,-1.5) {0};
        	\node[draw,circle] (1) at (-2,1) {1};
        	\node[draw,circle] (2) at (2,1) {2};
        	\node[draw,circle] (3) at (0,0) {3};
        	
        	\node[above=0.4] at (1) {0.9};
        	\node[above=0.4] at (2) {0.5};
        	\node[above=0.4] at (3) {0.1};
        	
        	\draw (0) -- (1) node[midway, below left] {0.6};
        	\draw (0) -- (2) node[midway, below right] {0.9};
        	\draw (0) -- (3) node[midway, right] {0.6};
        	\draw (1) -- (2) node[midway, above] {0.3};
        	\draw (1) -- (3) node[midway, above right] {0.1};
        	\draw (2) -- (3) node[midway, above left] {0.9};
        \end{tikzpicture}
        \caption{Example of a graph with send and crossing probabilities.}
        \label{fig: example graph}
    \end{figure}
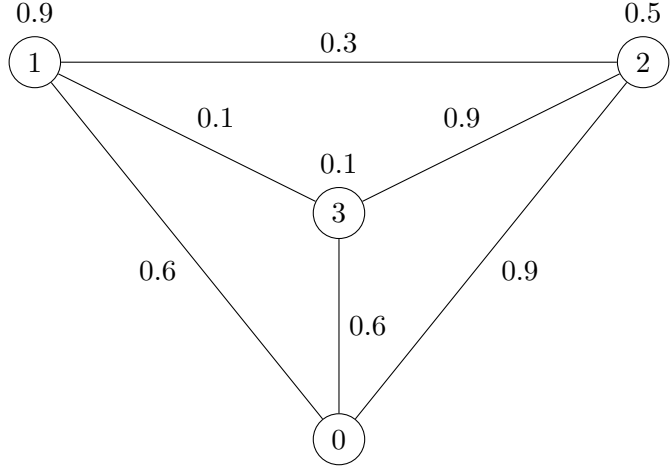

\subsection{NP-completeness}
Consider the following decision version of the reconnaissance problem above:

\begin{quote}
	Given an undirected connected graph $G = (\vertices, \edges)$, crossing probability $\probmove_{ij} \in [0,1]$ for every edge $\{i,j\} \in \edges$, transmission probability $\probsend_i \in [0,1]$ for every vertex $i \in \vertices$, a weight $\valinfo_i \geq 0$ that indicates the value of the information at $i \in \vertices$, and a real $r \in \mathbb{R}$, determine if there is a reconnaissance plan whose expected value is at least $r$.
\end{quote}

We will show that this decision problem is NP-complete. To this end, we first bound the length of a certificate for a `yes'-instance of this decision problem. This length depends primarily on the number of time periods the drone needs to travel in an optimal solution. 

\begin{theorem}
	The reconnaissance problem on the undirected connected graph $G = (\vertices, \edges)$ has an optimal solution consisting of at most $|\vertices|^2-1$ time periods.
	\label{prop:PlanLength}
\end{theorem}
\begin{proof}
	We restrict ourselves to feasible solutions where the drone does not travel in a cycle without transmitting or observing information. Removing such a cycle from the drone's walk keeps the plan feasible, and can never decrease the expected value of the transmitted information. This restriction makes the number of feasible solutions under consideration finite. It follows that this restricted problem has an optimal solution. This is also an optimal solution to the unrestricted problem, since adding cycles without transmitting or observing information cannot improve the objective value.
	
	Now let $n = |\vertices|$, and pick $\base \in \vertices$. Suppose that at some point in the reconnaissance plan, the drone has visited $k \in \{1, \dots, n\}$ unique vertices. (At the very start of the reconnaissance plan, $k = 1$, because the drone has only visited the base.) 
	
	To maximize the number of time periods before a new vertex is observed, the drone could first move to a location from where it transmits information. The drone can spend at most $k-1$ time periods traveling between the already visited vertices -- unless it travels in a cycle without transmitting or observing.
	
	After transmitting information, the drone can again spend at most $k-1$ time periods traveling between the already visited vertices before it makes a cycle without transmitting or observing a new vertex. Hence, there is an optimal solution with at most $2k-1$ time periods between the observations of the $k$-th and $(k+1)$-th unique vertex. In such an optimal solution, the number of time periods until all vertices are observed is at most $\sum_{k=1}^{n-1} (2k-1) = (n-1)^2$.
	
	Once all vertices are observed, the drone can take at most $n-1$ time periods to move to a transmission location, and at most another $n-1$ time periods to move back to the base. In sum, there is an optimal solution consisting of at most $(n-1)^2 + 2(n-1) = n^2-1$ time periods.
\end{proof}

It follows that the length of a certificate for a `yes'-instance of the decision problem is bounded by a polynomial of its input size. In other words, the reconnaissance problem lies in NP.

In fact, there are problem instances where any optimal solution consists of $|\vertices|^2 - O(|\vertices|)$ time periods, which shows that the dominant term in the bound of Theorem~\ref{prop:PlanLength} is tight.
\begin{proposition}
    The dominant term in the bound in Theorem~\ref{prop:PlanLength} is tight.
\end{proposition}
\begin{proof}
    We will construct a reconnaissance problem instance and lower bound the number of time periods in any optimal solution.

    Let $G = (\vertices, \edges)$ be a path graph such that $\base \in \vertices$ is an end point of the path, and let $n = |\vertices|$. Assume the non-base vertices are labeled such that the path is $(\base, 1, 2, \dots, n-1)$. Define the crossing probabilities $\probmove_{ij} = 1/\sqrt{n}$ for all $\{i,j\} \in \edges$, and transmission probabilities $\probsend_i = \mathbbm{1}_{i = \base}$ and information values $\valinfo_i = \mathbbm{1}_{i \neq \base}$ for all $i \in \vertices$.
    
    We will show that any optimal solution to this reconnaissance problem instance consists of at least $n^2 - n$ time periods.
 
	Let $\sigma_k$ be the probability of surviving until vertex $k \geq 1$ is visited for the first time, and let $\tau_k$ be the probability of surviving until the first base visit after the first visit to vertex $k$. The objective is then to maximize $\sum_{k=1}^{n-1} \tau_k$.
	
	As an induction hypothesis, suppose the drone does not carry any non-transmitted information when it visits a vertex $\ell \geq 1$ for the first time. At this point, $\sigma_k$ is fixed for all $k \leq \ell$ and $\tau_k$ is fixed for all $k \leq \ell - 1$.
	After observing the information at vertex $\ell$, the drone can either move back towards the base, or move forward.
	
	If the drone moves back to vertex $\ell-1$, this can only be optimal if it is part of a movement back to the base in $\ell$ time periods. If the drone does not move back to the base, or does so in more than $\ell$ time periods, the expected transmission value is decreased unnecessarily, since the crossing probabilities are smaller than one. 
	As a result of this strategy, the objective value would satisfy
	\begin{equation}
		\label{eq:MoveBackExpectedValue}
		\sum_{k=1}^{n-1} \tau_k \geq \sum_{k=1}^{\ell} \tau_k = \sum_{k=1}^{\ell-1} \tau_k + \sigma_\ell \left( \frac{1}{\sqrt{n}} \right)^\ell.
	\end{equation}
	
	If the drone moves forward to vertex $\ell+1$, it observes more information. It can gather at most $n-1-\ell$ units of information from the unvisited vertices with an index greater than $\ell$. To send this information, the drone has to travel back to the base, which means crossing at least $\ell+1$ edges. As a result of this strategy, the objective value would satisfy
	\begin{equation}
		\label{eq:MoveForwardExpectedValue}
		\sum_{k=1}^{n-1} \tau_k = \sum_{k=1}^{\ell-1} \tau_k + \sum_{k=\ell}^{n-1} \tau_k \leq \sum_{k=1}^{\ell-1} \tau_k + \sum_{k=\ell}^{n-1} \sigma_\ell \left( \frac{1}{\sqrt{n}} \right)^{1+\ell+1}.
	\end{equation}
	
	Since the lower bound from \eqref{eq:MoveBackExpectedValue} is strictly greater than the upper bound from \eqref{eq:MoveForwardExpectedValue} for any $\ell \geq 1$, we conclude that it is optimal for the drone to move back toward the base and transmit after observing a new piece of information. It then does not carry any non-transmitted information when arriving at vertex $\ell+1$, as we assumed. The total number of time periods required to visit all locations following this strategy is $\sum_{k=1}^{n-1} 2k = n^2 - n$.
\end{proof}

Having shown that the reconnaissance problem lies in NP, we now move on to showing it is NP-hard. To this end, we will provide a reduction from the Hamiltonian path problem with a fixed starting point. The following proposition shows that this problem is itself NP-complete.

\begin{proposition}
	Let $G = (\vertices, \edges)$ be an undirected graph such that $\base \in \vertices$. Then, the problem of deciding if $G$ contains a Hamiltonian path with starting point $\base$ is NP-complete.
\end{proposition}
\begin{proof}
	The problem lies in NP.
	Let $G' = (\vertices', \edges')$ be an undirected graph such that $\base, t \in \vertices'$. The problem of deciding whether there is a Hamiltonian path from $\base$ to $t$ in $G'$ is NP-complete, see e.g. Schrijver \cite[Corollary 8.11b]{schrijver2003combinatorial}. Now construct the graph $G = (\vertices, \edges)$ by adding a vertex $u$ to $G'$, connecting it only to $t$. Formally, $\vertices = \vertices' \cup \{u\}$ and $\edges = \edges' \cup \{\{t,u\}\}$. Then $G$ contains a Hamiltonian path with starting point $\base$ if and only if $G'$ contains a Hamiltonian path from $\base$ to $t$.
\end{proof}

We now reduce this fixed-start Hamiltonian path problem to a specific instance of the reconnaissance problem.
\begin{theorem}
	The reconnaissance problem is NP-complete.
\end{theorem}
\begin{proof}
	Let $G = (\vertices, \edges)$ be an undirected graph such that $\base \in \vertices$, and pick $q \in (0,1)$. Then set the crossing probabilities to $\probmove_{ij} = q$ for all $\{i,j\} \in \edges$, the transmission probabilities to $\probsend_i = 1$ for all $i \in \vertices$, and the information value to $\valinfo_i = 1$ for all $i \in \vertices \setminus \{\base\}$ (and $\valinfo_{\base} = 0$). Finally, set
	\begin{equation*}
		r = \frac{q - q^n}{1-q},
	\end{equation*}
	where $n = |\vertices|$. We claim that there is a reconnaissance plan with expected value at least $r$ if and only if there is a Hamiltonian path in $G$ starting at $\base$.
	
	If there is a Hamiltonian path starting at the base, we can construct a reconnaissance plan by letting the drone follow this path. At every vertex, the drone would observe and immediately transmit the information gathered there. The expected value of this plan is
	\begin{equation*}
		\sum_{k=1}^{n-1} q^k = \frac{q - q^n}{1-q} = r.
	\end{equation*}
	
	Conversely, suppose there is no Hamiltonian path starting at the base. Then, any reconnaissance plan will fall in one of two categories.
	\begin{enumerate}
		\item The drone does not visit all vertices. Visiting $m < n$ vertices will take at least $m-1$ crossings, making the expected value of such a plan at most
		\begin{equation*}
			\sum_{k=1}^{m-1} q^k = \frac{q - q^m}{1-q} < r.
		\end{equation*}
		\item The drone visits all vertices, but visits a vertex $u \in V$ a second time before it has visited all other vertices. The expected value of this reconnaissance plan is the sum of $n-1$ powers of $q$, where the exponents are distinct positive integers. If the second visit to $u$ occurs after $m < n-1$ crossings, the power $q^m$ does not appear in the computation of the plan's expected value. The expected value is then at most
		\begin{equation*}
			\sum_{k=1}^{m-1} q^k + q \sum_{k=m}^{n-1} q^k = \frac{q - q^m + q(q^m - q^n)}{1-q} < r.
		\end{equation*}
	\end{enumerate}
	We conclude that there is no reconnaissance plan with an expected value of at least $r$. The claim follows from Theorem~\ref{prop:PlanLength}.
\end{proof}

\section{Mixed-integer linear problem formulation for the autonomous reconnaissance problem}
This section will formulate the autonomous reconnaissance problem as a mixed-integer linear programming problem. To this end, we fix a time horizon $\timeperiods = \{1, \dots, \timemax\}$. In order to also find routes that require fewer time periods, we add the edge $\{\base,\base\}$ to the graph. In theory, we can take $\timemax=n^2-1$ as per Theorem \ref{prop:PlanLength}. Better estimates will be discussed in Section \ref{subsec:Performance}.

We will determine the drone's actions during each time period. A time period begins with the drone traveling from one location to another. If the drone survives this, the drone observes the information at its new location (if it was not observed before). Finally, the drone may or may not attempt to send the information it has gathered but not sent before. After each of these three stages, we can compute the drone's survival probability and the `expected transmission value' of the information the drone is carrying -- we will define these terms in more detail below. See Figure~\ref{fig:TimePeriodOverview} for an overview.

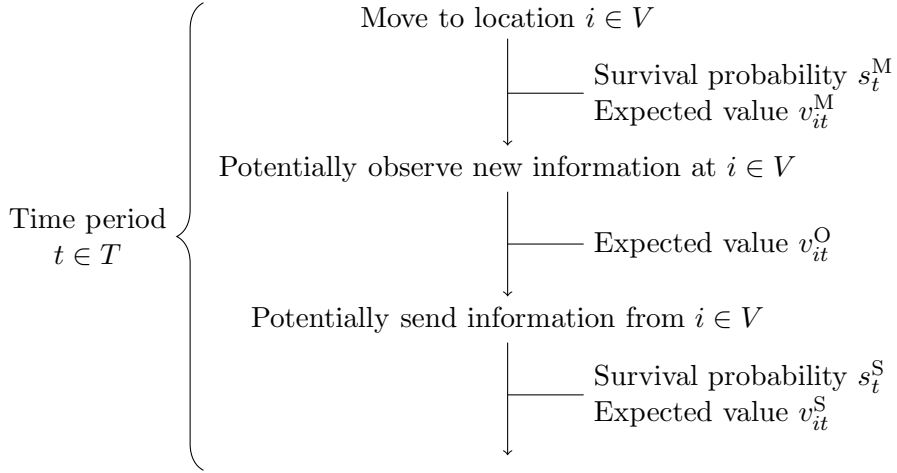
\begin{figure}[ht]
    \centering
    \begin{tikzpicture}
    	\node (A) at (0,0) {Move to location $i \in \vertices$};
    	\node (B) at (0,-2) {Potentially observe new information at $i \in \vertices$};
    	\node (C) at (0,-4) {Potentially send information from $i \in \vertices$};
    	\draw[->] (A) -- (B);
    	\draw[->] (B) -- (C);
    	\draw[->] (C) -- (0,-5.8);
    	
    	\node[align=left, anchor=west] (a) at (1,-1) {Survival probability $\survmove_t$ \\ Expected value $\expmove_{it}$};
    	\node[align=left, anchor=west] (b) at (1,-3) {Expected value $\expobs_{it}$};
    	\node[align=left, anchor=west] (c) at (1,-5) {Survival probability $\survsend_t$ \\ Expected value $\expsend_{it}$};
    	\draw (a) -- (0,-1);
    	\draw (b) -- (0,-3);
    	\draw (c) -- (0,-5);
    	
    	\draw[decorate, decoration={brace, amplitude=10pt, mirror}] (-4,0.2) -- node[midway, anchor=east, align=center, xshift=-10pt] {Time period \\ $t \in \timeperiods$} (-4,-6);
    \end{tikzpicture}
    \caption{Overview of a time period in the mixed-integer linear programming problem.}
    \label{fig:TimePeriodOverview}
\end{figure}

We first describe how we model the three stages of a time period. In Section~\ref{subsec:Moving}, we model the movement of the drone through the graph. Section~\ref{subsec:Observing} then describes the observation of the information at a node, while Section~\ref{subsec:Sending} discusses sending the information. After that, we describe how the drone's survival probability and expected transmission value can be computed in Section~\ref{subsec:Survival} and Section~\ref{subsec:ExpectedValue}, respectively. We state the final model in Section~\ref{subsec:FinalModel}.

\subsection{Moving the drone}
\label{subsec:Moving}
The route of the drone will be modelled by the decision variables
\begin{equation*}
	\varmove_{ijt} = \begin{cases}
		1 & \text{if the drone travels over $\{i,j\} \in \edges$ from $i$ to $j$ at time $t \in \timeperiods$}\\
		0 & \text{otherwise}.
	\end{cases}
\end{equation*}
Recall that $\{\base,\base\} \in \edges$, so the drone can also stay at the base in any time period.

We need a few constraints to let these decision variables describe a valid walk through the network. In every time period, the drone traverses exactly one edge, that is,
\begin{equation}
    \label{eq:OneEdgePerTimePeriod}
	\sum_{\{i, j\} \in \edges} \varmove_{ijt} = 1 \qquad \forall t \in \timeperiods.
\end{equation}
To make the drone start and end at the base, we require
\begin{equation}
    \label{eq:StartAndEndAtBase}
	\sum_{j: \{\base,j\} \in \edges} \varmove_{\base j1} = \sum_{i: \{i,\base\} \in \edges} \varmove_{i \base \timemax} = 1.
\end{equation}
In all other time periods, the drone starts at the location where it ended in the previous time period, meaning
\begin{equation}
    \label{eq:FlowConservation}
	\sum_{i: \{i,j\} \in \edges} \varmove_{ij,t-1} = \sum_{i: \{j,i\} \in \edges} \varmove_{jit} \qquad \forall j \in \vertices, t \in \timeperiods \setminus \{1\}.
\end{equation}

\subsection{Observing information}
\label{subsec:Observing}
When a drone arrives at a location for the first time, the information there is observed. We model this with the decision variables
\begin{equation*}
	\varobs_{it} = \begin{cases}
		1 & \text{if the drone observes the information at $i \in \vertices$ at time $t \in \timeperiods$}\\
		0 & \text{otherwise}.
	\end{cases} 
\end{equation*}
By allowing the information at any vertex to be observed only once, that is,
\begin{equation}
    \label{eq:OneObservation}
	\sum_{t \in \timeperiods} \varobs_{it} \leq 1 \qquad \forall i \in \vertices,
\end{equation}
the value of the information there can only be added to the accumulated information once. Since the objective is to maximize the expected value of the transmitted information, there will be an optimal solution where the information in every visited vertex $i \in \vertices$ is observed the first time $i$ is visited. After all, observing the information during later visits can only decrease the expected value of that information.

Of course, the information at $i \in \vertices$ can only be observed in time period $t \in \timeperiods$ if one also arrives in $i$ at time $t$. We model this by
\begin{equation}
    \label{eq:ObserveWhenThere}
	\varobs_{it} \leq \sum_{j: \{i,j\} \in \edges} \varmove_{ijt} \qquad \forall i \in \vertices, t \in \timeperiods.
\end{equation}

\subsection{Sending information}
\label{subsec:Sending}
After the new information has been observed, the drone may or may not transmit the information. This is modeled by the decision variables
\begin{equation*}
	\varsend_{it} = \begin{cases}
		1 & \text{if the drone transmits information from $i \in \vertices$ at time $t \in \timeperiods$}\\
		0 & \text{otherwise}.
	\end{cases}
\end{equation*}
Similar to observations, transmissions can only occur at a location in a certain time slot if one also arrives at that location in that time slot, that is,
\begin{equation}
    \label{eq:SendWhenThere}
	\varsend_{it} \leq \sum_{j: \{i,j\} \in \edges} \varmove_{ijt} \qquad \forall i \in \vertices, t \in \timeperiods.
\end{equation}

(One may argue that the values of the variables $x_{ijt}$ are already sufficient to determine the location of the drone, and that therefore there is no need to introduce transmission variables that are also indexed by the locations. This is true, but doing so would introduce additional non-linearities later on.)

\subsection{Survival probability}
\label{subsec:Survival}
The objective is to maximize the total expected value of the transmitted information. To compute the expected value of a transmission, we need to know the probability that the drone has survived until a time slot in which a transmission takes place.

We therefore introduce two new sets of decision variables:
\begin{itemize}
    \item $\survmove_t$ is the probability that the drone has survived from the start of the time horizon until after the \emph{moving} phase in time period $t \in \timeperiods$;
    \item $\survsend_t$ is the probability that the drone has survived from the start of the time horizon until after the \emph{sending} phase in time period $t \in \timeperiods$.
\end{itemize}
See Figure~\ref{fig:TimePeriodOverview} for an illustration. To ease notation, we also fix the parameter $\survsend_0 = 1$ to ensure the drone leaves the base with probability one.

The survival probability after moving is
\begin{equation}
    \label{eq:NonLinearSurvMove}
    \survmove_t = \survsend_{t-1} \sum_{\{i, j\} \in \edges} \probmove_{ij} \varmove_{ijt} \qquad \forall t \in \timeperiods,
\end{equation}
where $\probmove_{ij}$ is the probability of successfully moving over edge $\{i,j\} \in \edges$. Note that \eqref{eq:NonLinearSurvMove} is non-linear in the decision variables, because $\survsend_{t-1}$ is multiplied by $\varmove_{ijt}$. We can however linearize \eqref{eq:NonLinearSurvMove} if we replace $\survsend_{t-1} \varmove_{ijt}$ by the variable $\survsendvarmove_{ijt}$ subject to the constraints
\begin{subequations}
    \label{eq:DefineSurvsendVarmove}
    \begin{align}
        &\survsendvarmove_{ijt} \leq \survsend_{t-1} &&\forall \{i,j\} \in \edges, t \in \timeperiods\\
		&\survsendvarmove_{ijt} \leq \varmove_{ijt} &&\forall \{i,j\} \in \edges, t \in \timeperiods\\
		&\survsendvarmove_{ijt} \geq \survsend_{t-1} - (1 - \varmove_{ijt}) &&\forall \{i,j\} \in \edges, t \in \timeperiods\\
		&\survsendvarmove_{ijt} \geq 0 &&\forall \{i,j\} \in \edges, t \in \timeperiods.
    \end{align}
\end{subequations}
The definition of $\survmove_t$ from \eqref{eq:NonLinearSurvMove} can now be written as the linear equations
\begin{equation}
    \label{eq:DefineSurvMove}
    \survmove_t = \sum_{\{i, j\} \in \edges} \probmove_{ij} \survsendvarmove_{ijt} \qquad \forall t \in \timeperiods.
\end{equation}

Next, the survival probability after sending is equal to the survival probability after moving, unless one performs a transmission in the time period. That means
\begin{equation}
    \label{eq:NonLinearSurvSend}
    \survsend_t = \survmove_t \left( 1 - \sum_{i \in \vertices} (1 - \probsend_i) \varsend_{it} \right) \qquad \forall t \in \timeperiods,
\end{equation}
where $\probsend_i$ is the probability of a successful transmission at location $i \in \vertices$. Since \eqref{eq:NonLinearSurvSend} is also non-linear in the decision variables, we replace every product $\survmove_t \varsend_{it}$ by the variable $\survmovevarsend_{it}$ subject to the constraints
\begin{subequations}
    \label{eq:DefineSurvmoveVarsend}
    \begin{align}
        &\survmovevarsend_{it} \leq \survmove_t && \forall i \in \vertices, t \in \timeperiods\\
        &\survmovevarsend_{it} \leq \varsend_{it} && \forall i \in \vertices, t \in \timeperiods\\
        &\survmovevarsend_{it} \geq \survmove_t - (1 - \varsend_{it}) && \forall i \in \vertices, t \in \timeperiods\\
        &\survmovevarsend_{it} \geq 0 && \forall i \in \vertices, t \in \timeperiods.
    \end{align}
\end{subequations}
The definition of $\survsend_t$ from \eqref{eq:NonLinearSurvSend} can now be written as the linear equations
\begin{equation}
    \label{eq:DefineSurvSend}
    \survsend_t = \survmove_t - \sum_{i \in \vertices} (1 - \probsend_i) \survmovevarsend_{it} \qquad \forall t \in \timeperiods.
\end{equation}

\subsection{Expected transmission value}
\label{subsec:ExpectedValue}
We also track the value of the non-transmitted information the drone has gathered, multiplied by the survival probability of the drone. We call this the `expected transmission value' of the drone, since it captures the expected value of a transmission made in a certain time period.

There are three types of expected transmission values that we track by decision variables:
\begin{itemize}
    \item $\expmove_{it}$ is the expected transmission value of the drone after moving to location $i \in \vertices$ in time period $t \in \timeperiods$;
    \item $\expobs_{it}$ is the expected transmission value of the drone after potentially observing new information at location $i \in \vertices$ in time period $t \in \timeperiods$;
    \item $\expsend_{it}$ is the expected transmission value of the drone after potentially sending information from location $i \in \vertices$ in time period $t \in \timeperiods$.
\end{itemize}
See Figure~\ref{fig:TimePeriodOverview} for an illustration. As a matter of initialization, we fix the parameter $\expsend_{i0} = 0$ for all $i \in \vertices$.

By moving from $j \in \vertices$ to $i \in \vertices$ at time $t \in \timeperiods$, the expected transmission value gets multiplied by the probability of successfully moving over the edge $\{j,i\}$. 
In general, we get
\begin{equation}
    \label{eq:NonLinearExpMove}
	\expmove_{it} = \sum_{j: \{i,j\} \in \edges} \probmove_{ji} \varmove_{jit} \expsend_{j,t-1} \qquad \forall i \in \vertices, t \in \timeperiods.
\end{equation}
Since \eqref{eq:NonLinearExpMove} is non-linear in the decision variables, we replace every product $\varmove_{jit} \expsend_{j,t-1}$ by the variable $\expsendvarmove_{jit}$ subject to the constraints
\begin{subequations}
    \label{eq:DefineExpsendVarmove}
    \begin{align}
        &\expsendvarmove_{jit} \leq \expsend_{j,t-1} && \forall \{i,j\} \in \edges, t \in \timeperiods\\
        &\expsendvarmove_{jit} \leq \bigM \varmove_{jit} && \forall \{i,j\} \in \edges, t \in \timeperiods\\
        &\expsendvarmove_{jit} \geq \expsend_{j,t-1} - M(1 - \varmove_{jit}) && \forall \{i,j\} \in \edges, t \in \timeperiods\\
        &\expsendvarmove_{jit} \geq 0 && \forall \{i,j\} \in \edges, t \in \timeperiods,
    \end{align}
\end{subequations}
where we can take $\bigM = \sum_{i \in \vertices} \valinfo_i$.
The definition of $\expmove_{it}$ from \eqref{eq:NonLinearExpMove} can now be written as the linear equations
\begin{equation}
    \label{eq:DefineExpMove}
    \expmove_{it} = \sum_{j: \{i,j\} \in \edges} \probmove_{ji} \expsendvarmove_{jit} \qquad \forall i \in \vertices, t \in \timeperiods.
\end{equation}

Next, observing new information at location $i \in \vertices$ at time $t \in \timeperiods$ adds $\valinfo_i \survmove_{it}$ to the expected transmission value. In general, we therefore have
\begin{equation}
    \label{eq:NonLinearExpObs}
    \expobs_{it} = \expmove_{it} + \valinfo_i \survmove_{it} \varobs_{it} \qquad \forall i \in \vertices, t \in \timeperiods.
\end{equation}
Since \eqref{eq:NonLinearExpObs} is non-linear in the decision variables, we replace every product $\survmove_{it} \varobs_{it}$ by the variable $\survmovevarobs_{it}$ subject to the constraints
\begin{subequations}
    \label{eq:DefineSurvmoveVarobs}
    \begin{align}
        &\survmovevarobs_{it} \leq \survmove_{it} && \forall i \in \vertices, t \in \timeperiods\\
        &\survmovevarobs_{it} \leq \varobs_{it} && \forall i \in \vertices, t \in \timeperiods\\
        &\survmovevarobs_{it} \geq \survmove_{it} - (1 - \varobs_{it}) && \forall i \in \vertices, t \in \timeperiods\\
        &\survmovevarobs_{it} \geq 0 && \forall i \in \vertices, t \in \timeperiods.
    \end{align}
\end{subequations}
The definition of $\expobs_{it}$ from \eqref{eq:NonLinearExpObs} can now be written as the linear equations
\begin{equation}
    \label{eq:DefineExpObs}
    \expobs_{it} = \expmove_{it} + \valinfo_i \survmovevarobs_{it} \qquad \forall i \in \vertices, t \in \timeperiods.
\end{equation}

Finally, making a transmission sets the expected transmission value to zero. If no transmission is made, the expected transmission value does not change. This can be modelled by the constraints
\begin{subequations}
    \label{eq:DefineExpSend}
    \begin{align}
        & \expsend_{it} \leq \expmove_{it} && \forall i \in \vertices, t \in \timeperiods\\
        & \expsend_{it} \leq \bigM (1 - \varobs_{it}) && \forall i \in \vertices, t \in \timeperiods.
    \end{align}
\end{subequations}

\subsection{Final model}
\label{subsec:FinalModel}
As mentioned above, the objective is to maximize the expected value of the transmitted information, which would be
\begin{equation*}
    \sum_{i \in \vertices} \sum_{t \in \timeperiods} \probsend_i \expobs_{it} \varsend_{it}.
\end{equation*}
To linearize the objective, we replace every product $\expobs_{it} \varsend_{it}$ by the variable $\expobsvarsend_{it}$ subject to the constraints
\begin{subequations}
    \label{eq:DefineExpobsVarsend}
    \begin{align}
        &\expobsvarsend_{it} \leq \expobs_{it} && \forall i \in \vertices, t \in \timeperiods\\
        &\expobsvarsend_{it} \leq \bigM \varsend_{it} && \forall i \in \vertices, t \in \timeperiods\\
        &\expobsvarsend_{it} \geq \expobs_{it} - \bigM (1 - \varsend_{it}) && \forall i \in \vertices, t \in \timeperiods\\
        &\expobsvarsend_{it} \geq 0 && \forall i \in \vertices, t \in \timeperiods.
    \end{align}
\end{subequations}

In conclusion, the final model is
\begin{align*}
    \max\, & \sum_{i \in \vertices} \sum_{t \in \timeperiods} \probsend_i \expobsvarsend_{it}\\
    \text{subject to } & \varmove_{ijt} \in \{0,1\} && \forall \{i,j\} \in \edges, t \in \timeperiods\\
    & \varobs_{it}, \varsend_{it} \in \{0,1\} && \forall t \in \timeperiods,
\end{align*}
and constraints \labelcref{eq:OneEdgePerTimePeriod,eq:StartAndEndAtBase,eq:FlowConservation,eq:OneObservation,eq:ObserveWhenThere,eq:SendWhenThere,eq:DefineSurvsendVarmove,eq:DefineSurvMove,eq:DefineSurvmoveVarsend,eq:DefineSurvSend,eq:DefineExpsendVarmove,eq:DefineExpMove,eq:DefineSurvmoveVarobs,eq:DefineExpObs,eq:DefineExpSend,eq:DefineExpobsVarsend}.

\subsection{Performance}\label{subsec:Performance}

As was mentioned at the beginning of this section, a time horizon $\timemax$ needs to be set for the maximum number of time periods in a solution. From Proposition \ref{prop:PlanLength} we know that an optimal solution is guaranteed if we take $\timemax=|V|^2-1$. However, if all crossing probabilities are nonzero, often the optimal route is much shorter, as is supported by experimental results up to $|V|=10$ (some by the method in the next section).

The value of $\timemax$ is very influential on the run time of an implementation of the model. We have implemented the model in Python, using the Gurobi solver \cite{gurobi} on an 8 core Apple M2 processor with 8 GB of memory.

First, we ran the code for the graph in Figure \ref{fig: example graph} with 4 vertices. We fixed the time horizon at $\timemax=7$. The optimal solution, the same as described in Section \ref{subsec: first example}, was found in $0.29$ seconds.

We then generated a graph with random crossing probabilities on $6$ vertices. Setting $\timemax=9$, an optimal solution was found in $70.89$ seconds. For $\timemax=10$, the same optimal solution was found, but this took $736.62$ seconds (a bit over 12 minutes). For $\timemax=12$, the program ran out of memory after 538825 seconds (over 6 days). An attempt on a bigger machine with 12 CPU's and 80 GB of memory, also ran out of memory before finding a solution.

\section{Genetic algorithm}
\label{section: genetic algorithm}
Even though the above Mixed Integer Linear Program allows us to solve the reconnaissance problem to optimality, it quickly becomes too slow to be useful in practical application. The next section investigates whether we can replace the program with a heuristic, namely a genetic algorithm. This is an algorithm developed by Holland \cite{holland} in the 1970s and has been used in a lot of applications (see e.g. §14.4 in \cite{hillier} or §6 in \cite{yang} for a more in-depth summary of the method and its applications). Just as the name suggests, a genetic algorithm is based on the theory of evolution and tries to incorporate the ideas of survival of the fittest. The main idea is that at first, the algorithm is initialized with a list of feasible solutions, i.e. the first generation. All solutions in the generation are then ranked based on their expected value and this ranking is used to construct the next generation. The worst solutions of the generation are discarded: these genes do not survive. The best solutions are immediately copied into the next generation. This next generation is then filled with crossings from two solutions from the previous one. By repeating this process enough times and adding mutations to increase variation in the different generations, we expect that natural selection will lead to the optimal solution. A schematic depiction of how to construct new generations is given in Figure \ref{fig: genetic algorithm overview}. The advantage of this metaheuristic is that it allows us to optimize the route and transmission strategy simultaneously. We are able to cross both elements to obtain a new strategy that hopefully inherits the good traits of its parent strategies. Moreover, using specific crossing schemes and mutations we can make sure that all considered strategies in a generation are feasible and are not filled with routes that consist of non-existing edges.

\begin{figure}
    \centering
    \begin{tikzpicture}[scale=0.4]
    	\definecolor{mygreen}{rgb}{0.1,0.8,0.2}
    	\definecolor{mygray}{gray}{0.85}
    
    	\foreach \h in {0,...,3}{\draw[fill=mygreen] (0,-\h) rectangle (4,-\h-0.75);}
    	\foreach \h in {4,...,11}{\draw[fill=mygray] (0,-\h) rectangle (4,-\h-0.75);}
    	\foreach \h in {12,...,13}{\draw[fill=red] (0,-\h) rectangle (4,-\h-0.75);}
    	
    	\node[single arrow, draw, fill=mygreen, minimum height=2cm] at (7,-1.75) {};
    	\node[single arrow, draw, fill=mygreen, minimum height=2cm, rotate=-30] at (7,-4.5) {};
    	\node[single arrow, draw, fill=mygray, minimum height=2cm] at (7,-7.75) {};
    	
    	\foreach \h in {0,...,3}{\draw[fill=mygreen] (10,-\h) rectangle (14,-\h-0.75);}
    	\foreach \h in {4,...,13}{\draw[fill=mygray!85!mygreen] (10,-\h) rectangle (14,-\h-0.75);}
    	
    	\node[above] at (2,0) {Generation $n$};
    	\node[above] at (12,0) {Generation $n+1$};
    \end{tikzpicture}
    \caption{Overview of how to construct a new generation. First order the current generation from best to worst. Then discard the worst genes in the generation. This corresponds to the red part. Copy the best genes -- the green part -- directly in the new generation. The rest of the genes in the new generation is made of crossovers of genes that are either in the green part or the grey part of the previous generation.}
    \label{fig: genetic algorithm overview}
\end{figure}
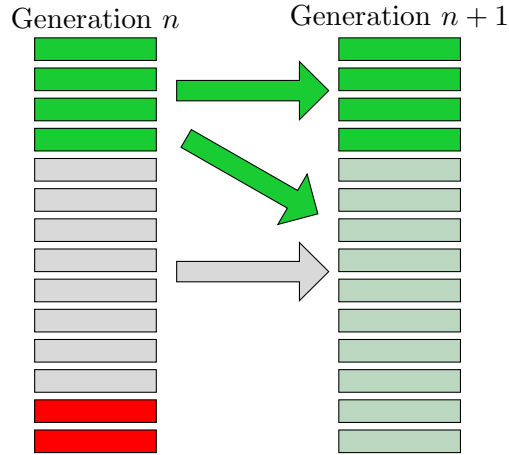

\subsection{Set-up of the algorithm}

\subsubsection{Initialization}
To initialize the algorithm, a list of random routes is generated. These routes are constructed by making a random walk of random length through the graph, starting from the base. Just as in the Linear Program, where the value of $\timemax$ had to be guessed, the upper bound $L_{\max}$ on the length of the walk also has to be chosen here. Similarly, we also choose a lower bound $L_{\min}$. Then we sample the actual length of the random walk uniformly at random from the set $\{L_{\min}, L_{\min} + 1, \dots, L_{\max} -1, L_{\max} \}$. Secondly, we repeatedly pick the next vertex in the walk uniformly at random from the vertices adjacent to the current position. If the last vertex of the walk is not the base vertex, the shortest path between the end of the walk and the base is added to the walk, using Dijkstra's algorithm \cite{dijkstra1959note}. As a rule of thumb, when considering a graph on $n$ vertices, we picked $L_{\min}$ to be approximately equal to $n-1$ and $L_{\max}$ to $n+100$. This might seem like a pretty high upper bound, especially because realistically, reconnaissance mission do not include more than 10 surveillance locations. But even though these longer routes are not necessarily better than the shorter ones -- at some point all information has been transmitted and making an additional loop doesn't change the expected value -- the final part of a route with useless and unnecessary walks can become useful after a crossover.

To quickly find a good send strategy for every route in the first generation, a local search algorithm is used. This is a strategy where at every vertex a transmission is made with probability $\pi$. We found that $\pi = 1/3$  gives a good starting point for the search algorithm. Next, we pick the best transmission strategy amongst all strategies at Hamming distance 1 of the current one. This process is repeated until there is no more improvement possible. Note that this method does not guarantee a global optimum. It is possible that the search gets stuck in a local maximum while there is a different send strategy with an even better expected value.
Based on their expected values we make a ranking of the strategies consisting of both a route and a transmission strategy.

\subsubsection{Cross-over}
In each iteration of the algorithm, a new generation is constructed based on the previous one. As already mentioned, we first copy the 10\% routes with the very best expected values directly in the next generation. Simultaneously, the 7.5\% worst routes are discarded. To construct a route in the next generation, two routes are picked uniformly at random from the best 92.5\%, i.e. all routes except the discarded ones. These two routes become the parents of a new route, hoping that the good genes of the parents will be inherited by the child. To cross these two parents, a vertex that is present in both routes is randomly chosen and the parent routes are sliced in two parts at the first occurrence of this vertex. From the one route we use the first part, from the other route the second part. The send strategies are sliced and put together in the exact same way. Again, this does not imply that the obtained send strategy is optimal for the newly constructed route. Moreover, it is possible that the only vertex in common is the base camp at the beginning and end of one of the routes. In that case, two new routes are picked and these are crossed instead.

This strategy keeps being repeated: the routes in this generation are ranked from best to worst and a new generation is constructed. After enough generations the hope is that eventually the globally optimal strategy will appear.

\subsubsection{Mutations}
To increase variation in the different generations, some of the newly crossed strategies are mutated. This means that they are slightly altered in their route or transmission strategy. 

Three mutations have been chosen to increase the number of considered strategies:

\begin{itemize}
    \item \textit{Added random walk.} Uniformly at random, one point that is not the last one is chosen in the route. After this point, a random walk of random length is added to the route. This random walk is generated in the same way as we did for the initialization. This mutation increases the length of the route.
    \item  \textit{Vertex flip.} One random vertex in the route is changed to a common neighbour of its preceding and succeeding vertex. Note that the existence of such a common neighbour is not guaranteed in a non-complete graph.
    \item  \textit{Send flip.} Uniformly at random one element in the send strategy is flipped: a 0 becomes 1 or vice versa.
\end{itemize}

These mutations are not always performed and for each strategy, there is at most one mutation performed. The \textit{added random walk} mutation is performed with a probability of 0.01. Since this mutation increases the length of the route, a new send strategy has to be constructed as well. This is done by discarding the old send strategy and performing the same local search as for the initialization for this new route.
The \textit{vertex flip} mutation is performed with a probability of 0.2 and the \textit{send flip} with a probability of 0.1. Since the length of the route doesn't change in either one, the send strategy is not changed after these mutations (this wouldn't make any sense for the \textit{send flip} either).

Having implemented these mutations, we can test whether they actually improve the genetic algorithm. At first, we considered the complete graph on 6 vertices, $K_6$. In this case, the genetic algorithm without mutations is able to find the same believed to be optimal strategy found using the MILP in 98\% of the runs. Hence there is no need for mutations, as they increase running time without having a lot of added value. On a non-complete graph on 10 vertices, this success rated dropped a bit to 90/100 times. However, the program returns the strategy in a couple of seconds and hence, is fit for the job. When considering $K_{10}$ finding the optimal strategy becomes more interesting. 

At first, we have ran the genetic algorithm on the graph where every vertex that is not the base camp holds one unit of information and the transmission probabilities (diagonal entries) and crossing probabilities (off-diagonal entries) are given by the following $10 \times 10$-matrix:  

$$
\begin{bmatrix}
         1 &0.95 &0.87 &0.93 &0.99 &0.96 &0.92 &0.88 &0.9 &0.93 \\ 0.95 &0.9 &0.86 &0.97 &0.93 &0.85 &0.82 &0.91 &0.93 &0.96 \\ 0.87 &0.86 &0.94 &0.92 &0.96 &0.98 &0.99 &0.82 &0.85 &0.91 \\
         0.93 &0.97 &0.92 &0.99 &0.87 &0.93 &0.9 &0.9 &0.89 &0.95 \\ 0.99 &0.93 &0.96 &0.87 &0.9 &0.94 &0.82 &0.85 &0.92 &0.9 \\ 0.96 &0.85 &0.98 &0.93 &0.94 &0.95 &0.91 &0.92 &0.91 &0.96 \\
         0.92 &0.82 &0.99 &0.9 &0.82 &0.91 &0.93 &0.98 &0.92 &0.93 \\ 0.88 &0.91 &0.82 &0.9 &0.85 &0.92 &0.98 &0.95 &0.99 &0.87 \\
         0.9 &0.93 &0.85 &0.89 &0.92 &0.91 &0.92 &0.99 &0.94 &0.85 \\ 0.93 &0.96 &0.91 &0.95 &0.9 &0.96 &0.93 &0.87 &0.85 &0.92
\end{bmatrix}.
$$

Running the algorithm without mutations yields the following best strategy:
\begin{align*}
    \text{Route: } &(0,4,0,5,2,6,7,8,1,3,9,3,0) \\
    \text{Send: }  &(0,0,1,0,0,0,1,0,0,1,0,1,1),
\end{align*} 
with value 7.305181. 

Out of 100 runs, this strategy is found 19 times. Compared to the previous graphs -- where the optimal strategy was almost always found -- this is not a high success rate. With this performance, a lot of runs are required to be pretty sure of finding the best solution. Say we want to have a probability greater or equal than 0.95. To find the optimal strategy assuming a success chance at every try of 19/100, we need to run the algorithm $h$ times, where 
$$ 1 - \left( \frac{81}{100} \right)^h \geq 0.95, $$
which implies that $h \geq 15$. While the algorithm still returns a strategy within a minute and this is not an infeasible number of runs, it would be more compelling to increase the success rate of a single try. 

To see how the genetic algorithm behaves, one can look at the best value of the strategies in each generation. This is depicted in Figure \ref{fig: zonder mutatie} for four runs of the algorithm. The $x$-axis shows the generation, the $y$-axis the best expected value. In all four runs, it is clear that the most progress is in the first generations of the algorithm and there is only one run that keeps improving the optimal value. The other runs quickly get stuck in a local maximum. The worst local maximum value where a run gets stuck on has expected value of around 7.11, which is quite far away from the believed to be optimal 7.305181.

    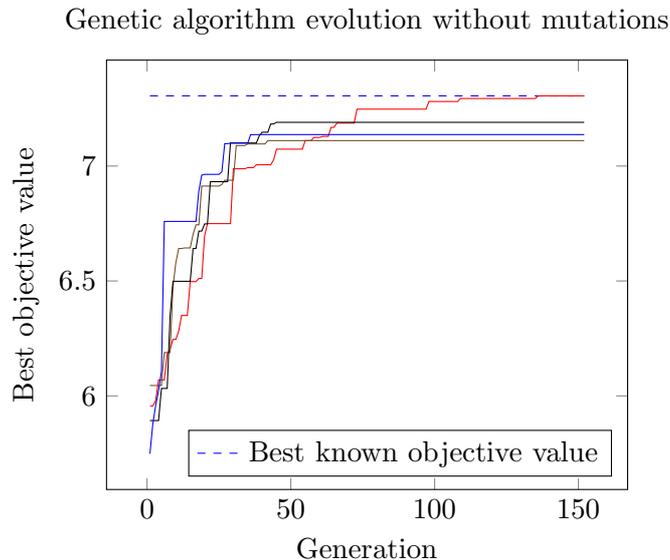
\begin{figure}[ht]
    \centering
    \begin{tikzpicture}
        \begin{axis}[xlabel={Generation}, ylabel={Best objective value}, title={Genetic algorithm evolution without mutations}, legend pos=south east]
    		\addplot+[no marks, dashed, domain=1:152] {7.30518118};
    		\addlegendentry{Best known objective value};
    		\foreach \i in {1,2,3,4}{
    			\addplot+[no marks] table [x=generation, y=run_\i] {Figures/Zonder_mutatie/coordinates_without_mutations.txt};
    		}
        \end{axis}
    \end{tikzpicture}
 
    \caption{Four runs of the genetic algorithm for $K_{10}$. The $x$-axis shows the generation, the $y$-axis the best expected value. There is only one run that finds the optimal strategy.}
    \label{fig: zonder mutatie}
\end{figure}

The plot shows that the algorithm gets stuck too often in local optima. This could be caused by a lack of variation in the different generations which means that other routes are left unexplored. Indeed, the 150th generation of the algorithm contains almost always exactly only one strategy that fills the entire list: the (local) maximum where it got stuck. Thus it is clear that the variation throughout the algorithm should be increased to hopefully find the optimal strategy more often. 

Rerunning the algorithm with all mutations shows their added value for larger graphs. Again, we have plotted four runs of the algorithm in Figure~\ref{fig: met mutatie}. This plot shows that adding mutations helps the algorithm to explore more strategies and end up in the best known strategy. 

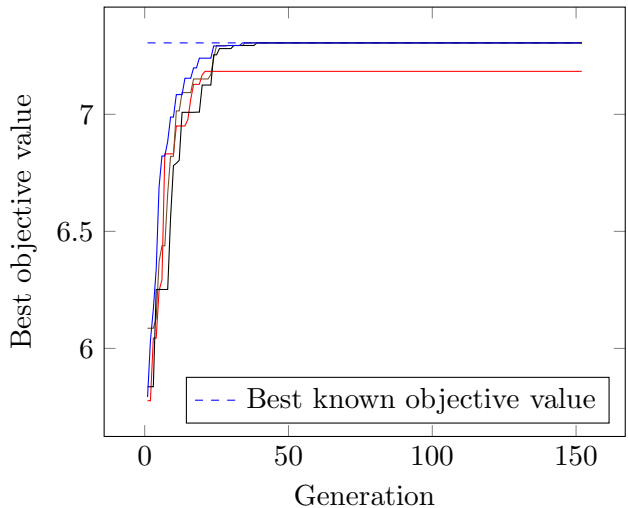
\begin{figure}[ht!]
    \centering
    \begin{tikzpicture}
        \begin{axis}[xlabel={Generation}, ylabel={Best objective value}, title={Genetic algorithm evolution with mutations}, legend pos=south east]
    		\addplot+[no marks, dashed, domain=1:152] {7.30518118};
    		\addlegendentry{Best known objective value};
    		\foreach \i in {1,2,3,4}{
    			\addplot+[no marks] table [x=generation, y=run_\i] {Figures/Met_mutatie/coordinates_with_mutations.txt};
    		}
        \end{axis}
    \end{tikzpicture}
    \caption{Four runs of the genetic algorithm with mutation. The $x$-axis shows the generation, the $y$-axis the best expected value.}
    \label{fig: met mutatie}
\end{figure}

\subsection{Comparison of the Different Versions}
Figure \ref{fig: met mutatie} shows that the mutations can enable the genetic algorithm to get out of local optima and eventually find the believed to be best planning. It mostly increases the speed by which a run is able to find the optimal strategy. But how often does such a successful run happen? To test this, the genetic algorithm was run a hundred times for $K_{10}$. Out of these 100 runs, the optimal value was found 61 times. Given the fact that one run takes less than a minute, it is feasible to run the genetic algorithm ten times. Assuming that the optimal value is found with a probability of 61/100 per run, the probability to find this strategy after ten runs is 0.9999. 

To put this result into perspective, the comparison of four different genetic algorithms can be made. The first considered algorithm does not contain any mutations, the second one only contains the \textit{vertex flip} mutation. Then there is the genetic algorithm with only the \textit{added random walk} mutation and lastly there is the genetic algorithm with all mutations combined as explained above. 

\begin{table}[ht]
\centering
\begin{tabular}{c|c}
    Genetic algorithm & \# of successful runs \\
    \midrule
    No mutations & 19/100 \\
    Send flip & 22/100 \\
    Vertex flip & 32/100 \\
    Added random walk & 55/100 \\
    Combination & 61/100
\end{tabular}
\caption{Performance of different versions of the genetic algorithm on hundred runs for $K_{10}$.}
\label{tab: 100 runs}
\end{table}

There is a clear difference in the performance between the different versions. Most mutations increase the probability of finding the optimal value. However, the send flip in itself does not clearly improve the algorithm. The vertex flip on its own does perform better, but 32/100 is still not amazing. Adding the random walk increases the odds of success more. In more than half of the runs the optimal planning was found. The genetic algorithm where three types of mutations are combined is still a little bit better, but since the genetic algorithm includes a lot of randomness, it is hard to determine whether this is actually better than only adding a random walk. Since both algorithms have no distinguishable speed difference, they seem interchangeable. 

In any case, the version of the algorithm with the combination of the mutations provides a tool to quickly and accurately find the believed to be optimal planning for larger graphs.

\section{Towards a Generalization: Multiple Drones}

So far, we have considered surveillance missions where only one drone is deployed. It makes sense to extend this scenario to missions with multiple drones. This section provides a generalization of the model for such missions and discusses options to adapt our algorithm to this model. First of all, similar modeling choices as before need to be made.

\begin{itemize}
    \item All drones start from the safe base camp, where there is no information to be retrieved. The base camp is also the last vertex in the routes of all drones.
    \item Routes and send strategies are determined before the mission, so there is no way of modifying the mission if a drone is intercepted by the enemy. Also, if this happens, it is assumed that the other drones can continue their mission.
    \item All drones are allowed to retrieve the same information. This means that the information is considered to be a picture taken at a location rather than a package that is retrieved.
    \item If information is successfully transmitted by a drone, the other drones can still send the same information. But this information should not be double counted. The expected value of transmitted information per vertex does increase if its information is sent multiple times, but can never exceed the total amount of information available at that specific vertex.
\end{itemize}

The last condition requires a new way of computing the expected value, as we cannot just look at the different transmissions per drone. Therefore we derive a new formula for the expected value that can be used for any number of drones.

While for the MILP, it was successful to focus on the transmission vertices, we will shift focus to how much information from each vertex is expected to reach the base camp. This enables us to compute the expected value for multi-drone missions. The expected amount of information that is safely transmitted can also be formulated as the sum of the expected fractions of each unit of information that is safely transmitted. If vertex $i$ contains $w_i$ units of information, then we compute $\mathbb{E}\left[ Y_i \right]$: the expected amount of information from vertex $i$ that reaches the base. In the case of one drone, this is equal to the product of $w_i$ with the probability that the drone safely reaches the first next transmission vertex and the probability that the transmission at this vertex is successful.

So, let $D_i$ denote the event that the drone successfully transmits the information that it collected at vertex $i$. Then, for a single drone mission,

\[ \mathbb{E}[X] = \sum_{i = 1}^{|V|} w_i \cdot \mathbb{P}(D_i).\]

This formula can be extended to multiple drones, but we will require some more notation. Suppose that we have $\ell$ drones and let $D_{iv}$ be the binary random variable that equals 1 if drone $i$ succeeds in transmitting the information from vertex $v$. Moreover, let $I \subseteq [\ell] = \{1, \dots, \ell\}$. 
Then for a given $v$, we are interested in 
\begin{align*}
    \mathbb{P} \left (\bigcup_{i \in [\ell]} D_{iv} \right) &= \sum_{\substack{I \subseteq [\ell] \\ I \neq \emptyset}} (-1)^{|I|-1} \cdot \mathbb{P} \left( \bigcap_{i \in I} D_{iv} \right)\\
    &= \sum_{\substack{I \subseteq [\ell] \\ I \neq \emptyset}} (-1)^{|I|-1} \cdot \prod_{i \in I} \mathbb{P} \left( D_{iv} \right),
\end{align*}
where we have used the inclusion-exclusion principle and independence of probabilities.

Hence, this leads to a formula for the expected value $\mathbb{E}\left[ Y_i \right]$:

$$\mathbb{E}\left[ Y_i \right] = w_i \cdot \sum_{\substack{I \subseteq [\ell] \\ I \neq \emptyset}} (-1)^{|I|-1} \prod_{i \in I} \mathbb{P} \left( D_{iv} \right).$$

Linearity of expectation provides an alternative formula for the total expected value:

$$\mathbb{E}[X] = \sum_{i=1}^{|V|}\mathbb{E}\left[ X_i \right]. $$

\subsection{Adaptation of the Genetic Algorithm}

As the MILP formulation was already slow for a single drone mission, it seems hard to generalize it to a practically relevant program for any number of drones. Therefore we immediately present a generalization of the genetic algorithm. As we have a formula for the expected value of a multi-drone strategy, we are able to rank the different strategies in the genetic algorithm. The initialization of the algorithm is very similar as before. At random, the required number of routes is created. These routes are again allowed to be very long, i.e $L_{\max}$ remains $n+100$ for a graph on $n$ vertices. The transmission strategy is determined by a local search algorithm on the total expected value with a randomly generated starting point as in Section \ref{section: genetic algorithm}. This means that we only flip one transmission at a time combined over the send strategies of all drones. Concerning the cross-over, the same method as for the single drone mission is kept. However, at first we randomly permute all routes of the drones in the mission to then pairwise cross them over to obtain a new multi drone strategy. 

Also the mutations have the same flavour as in the single drone case. First of all, there is the \textit{added random walk}. This is the same mutation as described for the single drone: one of the routes is picked and a random walk is added somewhere in the walk. This continues to be a useful mutation to increase the variation in the genetic program. But it slows down the algorithm, because it requires a local search for the optimal send strategy in terms of the expected value for multiple drones. Therefore this mutation is only actualized with a low probability. If a generation consists of $\kappa$ strategies, the mutation probability is set at 2/$\kappa$. 

A second important remark that can be made is that because of the formulation of the expected value, the different drones are incentivized to still travel to all vertices of the graph and try to gather information everywhere. However, it does make sense that the drones fly more or less in `opposite directions' through the graph. In this case, more vertices have a high probability of their information being successfully transmitted by at least one of the drones. This inspired the \textit{reversed} mutation. In this mutation, exactly one of the routes is completely reversed. The same holds for the corresponding send strategy. This mutation is computationally cheaper than the other one so this is actualized with 20\% probability.

\subsection{Performance of the Multi-Drone Genetic Algorithm}

After having adapted the genetic algorithm to the multiple drone scenario, we have tested this algorithm on multiple graphs using two drones. On a graph with six vertices and a sparse graph on ten vertices, the genetic algorithm remained fast and reliable. For example for the complete graph on 6 vertices given by the matrix

$$\begin{bmatrix}
1&0.97&0.81&0.97&0.95&0.96\\
0.97&0.93&0.92&0.87&0.89&0.87\\
0.81&0.92&0.96&0.81&0.98&0.93\\
0.97&0.87&0.81&0.85&0.93&0.93\\
0.95&0.89&0.98&0.93&0.91&0.89\\
0.96&0.87&0.93&0.93&0.89&0.90
\end{bmatrix},$$

\noindent
the returned solution is

\begin{center}
\begin{tabular}{c c c}
\multirow{2}{4em}{Drone 1} & Route: & $(0,3,0,5,0,4,2,1,0)$ \\
& Send: & $(0,0,1,0,1,0,1,0,1)$ \\
\\
\multirow{2}{4em}{Drone 2} & Route: & $(0,1,0,4,2,5,0,3,0)$ \\
& Send: & $(0,0,1,0,1,0,1,0,1)$\\
\end{tabular}
\end{center}

\noindent
and increased the expected value from 4.115090 for a single drone to 4.859738 for two drones out of 5 possible units of information. Note that indeed, the structure of the two routes is approximately opposite. But on the example of $K_{10}$, the current algorithm is not sufficient. The algorithm becomes too slow for practical use and does not consistently return the same optimal strategy. The best strategy so far is given by 
\begin{center}
\begin{tabular}{c c c}
\multirow{2}{4em}{Drone 1} & Route: & $(0,5,2,6,7,8,7,6,2,5,9,3,1,0,4,0)$\\
& Send: & $(0,0,0,0,0,0,1,0,0,0,0,1,0,1,0,1)$\\
\\
\multirow{2}{4em}{Drone 2} & Route: & $(0,4,0,1,3,9,5,2,6,7,8,7,6,2,4,0)$ \\
& Send: & $(0,0,1,0,1,0,0,0,0,1,0,1,0,0,0,1)$ \\
\end{tabular}
\end{center}
with expected value 8.653276, but was found in less than 10\% of the runs. The best solution for one drone had expected value 7.305181, showing that it does pay off for larger missions to use multiple drones, even if that strategy is not necessarily optimal. However, as a further research, it would be interesting to further improve the used methods to solve this problem for even larger instances and find the optimal solution more easily.

\bibliographystyle{plain}
\bibliography{sources}

\begin{thebibliography}{10}

\bibitem{albert2017UAV}
Anders Albert, Frederik Leira, and Lars Imsland.
\newblock Uav path planning using milp with experiments.
\newblock {\em Modeling, Identification and Control}, 38:21--32, 01 2017.

\bibitem{alotaibi2018unmanned}
Kamil~A Alotaibi, Jay~M Rosenberger, Stephen~P Mattingly, Raghavendra~K Punugu, and Siriwat Visoldilokpun.
\newblock Unmanned aerial vehicle routing in the presence of threats.
\newblock {\em Computers \& Industrial Engineering}, 115:190--205, 2018.

\bibitem{Marine}
Arthur~P. Brill~Jr.
\newblock Dragon drone deployment tests corps' {UAV} capabilities.
\newblock {\em Sea Power}, 1998.

\bibitem{DroneMarket}
Teal~Group Corporation.
\newblock 2022/2023 {W}orld {C}ivil {U}nmanned {A}erial {S}ystems {M}arket {P}rofile \& {F}orecast, 11 2023.

\bibitem{dasdemir2020flexible}
Erdi Dasdemir, Murat K{\"o}ksalan, and Diclehan~Tezcaner {\"O}zt{\"u}rk.
\newblock A flexible reference point-based multi-objective evolutionary algorithm: An application to the uav route planning problem.
\newblock {\em Computers \& operations research}, 114:104811, 2020.

\bibitem{dijkstra1959note}
EW~Dijkstra.
\newblock A note on two problems in connexion with graphs.
\newblock {\em Numerische Mathematik}, 1:269--271, 1959.

\bibitem{Vietnam}
Tom Engelhardt.
\newblock Remotely {P}iloted {W}ar.
\newblock {\em Foreign Policy in Focus}, 2012.

\bibitem{fink2019globally}
Wolfgang Fink, Victor~R Baker, Alexander J-W Brooks, Michael Flammia, James~M Dohm, and Mark~A Tarbell.
\newblock Globally optimal rover traverse planning in 3d using dijkstra’s algorithm for multi-objective deployment scenarios.
\newblock {\em Planetary and Space Science}, 179:104707, 2019.

\bibitem{gurobi}
{Gurobi Optimization, LLC}.
\newblock {Gurobi Optimizer Reference Manual}, 2023.

\bibitem{han2021satellite}
Chen Han, Aijun Liu, Kang An, Haichao Wang, Gan Zheng, Symeon Chatzinotas, Liangyu Huo, and Xinhai Tong.
\newblock Satellite-assisted uav trajectory control in hostile jamming environments.
\newblock {\em IEEE Transactions on Vehicular Technology}, 71(4):3760--3775, 2021.

\bibitem{hillier}
Frederick~S. Hillier and Gerald~J. Lieberman.
\newblock {\em Introduction to Operations Research}.
\newblock McGraw-Hill Education, 11th edition, 2021.

\bibitem{HybridConcept}
Frank~G. Hoffman.
\newblock Hybrid {T}hreats: {R}econceptualizing the {E}volving {C}haracter of {M}odern {C}onflict.
\newblock {\em Strategic Forum}, 2009.

\bibitem{holland}
John~H Holland.
\newblock {\em Adaptation in natural and artificial systems: an introductory analysis with applications to biology, control, and artificial intelligence}.
\newblock MIT press, 1992.

\bibitem{jia2020dynamic}
Tao Jia, Shaohuan Han, Ping Wang, Wenyuan Zhang, and Yawu Chang.
\newblock Dynamic obstacle avoidance path planning for uav.
\newblock In {\em 2020 3rd International Conference on Unmanned Systems (ICUS)}, pages 814--818. IEEE, 2020.

\bibitem{jiang2018routing}
Jinfang Jiang and Guangjie Han.
\newblock Routing protocols for unmanned aerial vehicles.
\newblock {\em IEEE Communications Magazine}, 56(1):58--63, 2018.

\bibitem{HybridUkraine}
Sarah~J. Lohman.
\newblock {\em What Ukraine Taught NATO about Hybrid Warfare}.
\newblock Strategic Studies Institute and US Army War College Press, Carlisle Barracks, Pennsylvania, USA, 2022.

\bibitem{lu2023uav}
Yuxi Lu, Wu~Wen, Kostromitin~Konstantin Igorevich, Peng Ren, Hongxia Zhang, Youxiang Duan, Hailong Zhu, and Peiying Zhang.
\newblock Uav ad hoc network routing algorithms in space--air--ground integrated networks: Challenges and directions.
\newblock {\em Drones}, 7(7):448, 2023.

\bibitem{quadt2024dealing}
Thomas Quadt, Roy Lindelauf, Mark Voskuijl, Herman Monsuur, and Boris {\v{C}}ule.
\newblock Dealing with multiple optimization objectives for uav path planning in hostile environments: A literature review.
\newblock {\em Drones (2504-446X)}, 8(12), 2024.

\bibitem{roberge2012comparison}
Vincent Roberge, Mohammed Tarbouchi, and Gilles Labont{\'e}.
\newblock Comparison of parallel genetic algorithm and particle swarm optimization for real-time uav path planning.
\newblock {\em IEEE Transactions on industrial informatics}, 9(1):132--141, 2012.

\bibitem{sajid2022routing}
Mohammad Sajid, Himanshu Mittal, Shreya Pare, and Mukesh Prasad.
\newblock Routing and scheduling optimization for uav assisted delivery system: A hybrid approach.
\newblock {\em Applied Soft Computing}, 126:109225, 2022.

\bibitem{schrijver2003combinatorial}
Alexander Schrijver.
\newblock {\em Combinatorial optimization: polyhedra and efficiency}, volume~A.
\newblock Springer, 2003.

\bibitem{song2019survey}
Baoye Song, Gaoru Qi, and Lin Xu.
\newblock A survey of three-dimensional flight path planning for unmanned aerial vehicle.
\newblock In {\em 2019 Chinese Control And Decision Conference (CCDC)}, pages 5010--5015. IEEE, 2019.

\bibitem{DronesWarfare}
Frank~Christian Sprengel.
\newblock Drones in hybrid warfare: Lessons from current battlefields ({W}orking {P}aper).
\newblock {\em Hybrid CoE}, 2021.

\bibitem{yang}
Xin-She Yang.
\newblock {\em Nature-inspired optimization algorithms}.
\newblock Academic Press, 2020.

\bibitem{yu2019algorithms}
Kevin Yu, Ashish~Kumar Budhiraja, Spencer Buebel, and Pratap Tokekar.
\newblock Algorithms and experiments on routing of unmanned aerial vehicles with mobile recharging stations.
\newblock {\em Journal of Field Robotics}, 36(3):602--616, 2019.

\end{thebibliography}

\end{document}